\documentclass[11pt]{article}
\usepackage{amsmath}
\usepackage{dsfont}
\usepackage{mathrsfs}
\usepackage{amsmath,amssymb}
\usepackage{amsfonts}
\usepackage{hyperref}
\usepackage{amsthm}
\usepackage{graphicx}
\usepackage{subfigure}
\usepackage{xcolor}

\hfuzz=\maxdimen
\newtheorem{lemma}{Lemma}[section]
\newtheorem{definition}[lemma]{Definition}

\newtheorem{theorem}[lemma]{Theorem}
\newtheorem{corollary}[lemma]{Corollary}

\numberwithin{equation}{section}

\DeclareFixedFont{\Acknowledgment}{OT1}{cmr}{bx}{n}{14pt}
\begin{document}

\title{\bf Kazdan-Warner equation on infinite graphs}
\author{Huabin Ge and Wenfeng Jiang}
\maketitle

\begin{abstract}
We concern in this paper the graph Kazdan-Warner equation
\begin{equation*}
\Delta f=g-he^f
\end{equation*}
on an infinite graph, the prototype of which comes from the smooth Kazdan-Warner equation on an open manifold. Different from the variational methods often used in the finite graph case, we use a heat flow method to study the graph Kazdan-Warner equation. We prove the existence of a solution to the graph Kazdan-Warner equation under the assumption that $h\leq0$ and some other integrability conditions or constrictions about the underlying infinite graphs.
\end{abstract}

\section{Introduction}
The smooth Kazdan-Warner equation gives a description of the conformal deformation of a smooth metric $g$ on a $2$-dimensional closed Riemannian manifold $(M, g)$. Let $\widetilde{g}=e^{2f}g$ be a conformal deformation of the smooth metric $g$. To find a smooth function $\widetilde{K}$ as the Gaussian curvature of $\widetilde{g}$, one needs to solve the nonlinear elliptic equation
$$\Delta_gf=K-\widetilde{K}e^{2f}.$$
By a parameter transformation of $f$ to $u$, the above equation takes the following form
\begin{equation}
\label{equ-intro-c}
\Delta_gu=c-he^{u},
\end{equation}
where $c$ is a constant, and $h$ is some prescribed function, with neither $c$ nor $h$ depends on the geometry of $(M, g)$. In the fundamental and pioneering work of Kazdan and Warner \cite{KZ-1}, they discussed the equation (\ref{equ-intro-c}) and gave almost completely characterizations to the solvability of (\ref{equ-intro-c}). In case $c<0$, they proved that there is a threshold $c_h<0$ such that the above equation has a solution if $c\in(c_h, 0)$, while it has no solution if $c<c_h$. However, the existence of a solution also in the critical case where $c=c_h$ was proved by Chen and Li \cite{CL}. Kazdan and Warner \cite{KZ-2} also discussed the solvability of (\ref{equ-intro-c}) on certain non-compact two dimensional manifolds.

The literature on the Kazdan-Warner equation is huge. However, in a very recent paper \cite{GLY}, Grigor'yan, Lin and Yang have studied this equation on a finite graph. In particular, in the regime where $c<0$, their result is essentially the same as in the manifold case of Kazdan and Warner. And also there is a solution in the critical case where $c=c_h$, which was proved by the first author of this paper Ge \cite{G}. Following Grigor'yan, Lin and Yang's idea, Ge \cite{G2} studied the $p$-th Kazdan-Warner equation $\Delta_pu=c-he^u$ on a finite graph.

Inspired by these works, we concern in this paper the Kazdan-Warner equation
\begin{equation}
\label{equ-intro-kw-graph}
\Delta f=g-he^f
\end{equation}
on an infinite graph, the prototype of which comes from the smooth Kazdan-Warner equation on an open manifold. The problem in the finite graph case
is somewhat simplified by the fact that function spaces are finite dimensional. Different from the variational methods often used in the finite graph case, we use a heat flow method (see Wang-Zhang \cite{WZ}) to study the Kazdan-Warner equation. We prove the existence of a solution to the Kazdan-Warner equation (\ref{equ-intro-kw-graph}) on an infinite graph under the assumption that $h\leq0$ and some other integrability conditions or constrictions about the underlying infinite graphs.

The paper is organized as follows. In section 2, we state the basic settings and main theorems. In section 3, we prove our main theorem, i.e. Theorem \ref{thm-main} in three steps.

\section{Settings and main results}
Let $G= (V, E)$ be a  graph, where $V$ denotes the vertex set and $E$ denotes the edge set. The graph $G$ is called connected, if one can not find a subset $V_1\subset V$, such that there are no edges between $V_1$ and $V_1^c$. A vertex measure $\mu$ is a map $\mu:V \to (0, \infty)$. An edge measure $w$ is a symmetric map $w: E \to (0, \infty) $ on G, where symmetric means, $w_{xy} = w_{yx}$ for each edge $x\thicksim y$. We call $G$ locally finite, if at each vertex $x$,  the set $\{y|y \thicksim x\}$ is finite. Throughout this paper, we denote $C(G,f,\ldots)$ as some positive constant depending only on the information of $G$, $f,\ldots$. Note that the information of $G$ contains $V,E,\mu$ and $w$. In the following, we always assume that $G$ is connected, infinite and locally finite.

For a function $f:V\to \mathbb{R}$, the graph Laplacian (with respect to $\mu,w$) is defined by
$$\Delta f(x)=\frac{1}{\mu(x)}\sum_{y\thicksim x}w_{xy}\big(f(y)-f(x)\big).$$
The integral of $f$ over $V$ (with respect to the vertex measure $\mu$) is defined by
$$\int_{V}fd\mu=\sum_{x\in V}f(x)\mu(x).$$
Note $f$ may not be integrable generally. Denote $L^p(V)$ as the space of all $p$-th integrable functions on $V$ (with respect to the vertex measure $\mu$). The absolute value of the gradient of $f$ is defined by
$$|\nabla f|(x)=\left (\frac{1}{2\mu(x)}\sum_{y\thicksim x}w_{xy}\big(f(x)-f(y)\big)^2 \right)^{1/2}.$$
If $|\nabla f|\in L^2(V)$, then a direct calculation gives the following Green formula
\begin{equation*}
\label{general-1}
\int_V |\nabla f|^2d\mu=-\int_V f \Delta fd\mu.
\end{equation*}

Now we consider the Kazdan-Warner equation on the graph $G$
\begin{equation}
\label{kz-def-t-1}
\Delta f+he^f-g=0,
\end{equation}
where $g$ and $h$ are known functions on $V$. We say $f: V\to \mathbb{R}$ is a global solution of (\ref{kz-def-t-1}), if $\Delta f(x)+h(x)e^{f(x)}-g(x)=0$ at each vertex $x\in V$. 

\begin{definition}
(full subgraph)
A full subgraph $G'\subset G$ is a subgraph of $G$, such that (1) for a pair of vertices $x$, $y$ in $G'$, $x\thicksim y$ in $G'$ if and only if $x\thicksim y$ in $G$, and (2) the vertex measure and edge measure in $G'$ are exactly the restriction of $\mu$ and $w$ for $G$.
\end{definition}


\begin{definition}
(exhaustion of a graph)
An exhaustion $\{G_k\}$ of the graph $G$ is a sequence of finite full subgraphs $G_1, G_2,\cdots$, with vertex sets $V_1\subset V_2\subset\cdots$ so that $V=\bigcup\limits_{k}V_k$.
\end{definition}

For each $G_k$ in an exhaustion of $G$, we set $\overline{V_k}=V_k \cup \{x|x\in G, x\thicksim y \text{ for some } y \in G_k \},$
and denote by $\overline{G_k}$ as its corresponding full subgraph. Moreover, we set $\partial  V_k=\overline{V_k}\setminus V_k$.

\begin{definition}
(Cheeger graph)
We say a graph $G$ with a vertex measure $\mu$ and an edge measure $w$ is a Cheeger graph, if there exists an exhaustion $\{G_k\}$ of $G$, such that for each $k$ and for each function $f:\overline{V_k}\to \mathbb{R}$ with $f|_{\partial {V_k}}\equiv 0$, the following Cheeger inequality holds
\begin{equation}
\label{cheeger-inequality}
\|\nabla f\|_{L^{2}(\overline{V_k})}>C(G)\|f\|_{L^{2}(\overline{V_k})},
\end{equation}
where $C(G)$ is a constant depending only on the information of $G$ and measure $\mu,w$.
\end{definition}

The inequality (\ref{cheeger-inequality}) is the graph theory version of Cheeger's inequality \cite{C} of eigenvalues of graph Laplacians. As is known, the curvature-dimension type inequality $CD(m,K)$ often implies a lower bound of the nonzero eigenvalues of graph Laplacians, see \cite{BCLL} for example. In addition, the curvature-dimension type inequality, the eigenvalue estimations, the isoperimetric constants, the Sobolev constants and the Cheeger-type inequalities are closely related to each other. They all reveal some geometric information of a manifold (or a graph). By now, there is a vital interest in the study of these objects on finite or infinite graphs. See\cite{APP}\cite{BSS}\cite{BHJ}\cite{BKWR}\cite{CGY}\cite{D}\cite{KMD}\cite{PRT} for more related works.

Generally, the Kazdan-Warner equation may not have a global solution if no conditions are imposed on $g$, $h$ and $G$. We pose two sufficient conditions in the paper, each of which guarantees the existence of a global solution of the Kazdan-Warner equation (\ref{kz-def-t-1}).\\[-6pt]


\noindent \emph{(C-1)}. $h\in L^1(V)$, $g\leq h<0$, and $\int_V\left( \frac{g}{h}\right)^2|h|d\mu<+\infty.$\\[-6pt]


\noindent \emph{(C-2)}. $G$ is Cheeger, $g\in L^2(V)$, $h\in L^1(V)$ and $h\leq0$.\\[-10pt]

\begin{theorem}
\label{thm-main}
The Kazdan-Warner equation (\ref{kz-def-t-1}) has a global solution on $G$ under either condition (C-1) or (C-2).
\end{theorem}

\begin{corollary}
Then Poisson equation $\Delta f=g$ with $g\in L^2(V)$ has a global solution on a Cheeger graph $G$.
\end{corollary}

In case $g$ is a constant $c$, The Kazdan-Warner equation (\ref{kz-def-t-1}) changes to
\begin{equation}
\label{equ-kz-c-equat}
\Delta f=c-he^f.
\end{equation}
Obviously, Theorem \ref{thm-main} implies the following three corollaries.

\begin{corollary}
Assume $c\leq h<0$, and $h$, $h^{-1}\in L^1(V)$, then the equation (\ref{equ-kz-c-equat}) has a global solution on $G$.
\end{corollary}

\begin{corollary}
\label{coro-finit-volume}
Assume $G$ has finite volume, i.e. $\int_Vd\mu<+\infty$, then (\ref{equ-kz-c-equat}) has a global solution if $c\leq h\leq-\epsilon$, where $\epsilon$ is a positive constant.
\end{corollary}

\noindent \emph{Remark 1.}
It appears that all existing results about the solvability of some nonlinear equations on $G$ need $``\mu$-noncollapsing", i.e. there is a positive constant $\delta$ so that $\mu(x)\geq\delta$ for each $x\in V$. Corollary \ref{coro-finit-volume} seems the first one dealing with $``\mu$-collapsing" case.

\begin{corollary}
\label{coro-c=0}
Assume $G$ is a Cheeger graph, then the following equation
\begin{equation}
\label{kw-equ-c=0}
\Delta f=-he^f
\end{equation}
has a global solution on $G$ if $h\in L^1(V)$ and $h\leq0$.
\end{corollary}

\noindent \emph{Remark 2.}
Geometrically, $c=0$ means the underlying manifold is everywhere flat. Hence in the smooth case, the equation (\ref{kw-equ-c=0}) is in fact the smooth Kadan-Warner equation on $\mathbb{R}^2$. In this sense, Corollary \ref{coro-c=0} deals with the solvability of the graph Kazdan-Warner equation on a Cheeger graph drawing on $\mathbb{R}^2$. Note on $\mathbb{R}^2$, Sattinger \cite{Sat} proved that (\ref{kw-equ-c=0}) has no solution if $h(x)\leq-\frac{C}{\|x\|^2}$ at infinity, while Ni \cite{ni-3} proved that (\ref{kw-equ-c=0}) possesses infinitely many solutions if $0\geq h(x)\geq-\frac{C}{\|x\|^l}$ at infinity for some $l>2$. Hence our assumption that $h\in L^1(V)$ seems sharp for the existence of a solution to (\ref{kw-equ-c=0}) on a planar graph.

\section{Proof of Theorem \ref{thm-main}}
In this section, we prove the existence of a global solution for the Kazdan-Warner equation under either condition (C-1) or condition (C-2). The proof is divided into three steps.

In Step 1, we transform the Kazdan-Warner equation (\ref{kz-def-t-1}) to an associated equation on $G_k$ for each $k$, where $\{G_k\}$ is an exhaustion of $G$.

In Step 2, we show that for each $k$, the associated equation has a solution $f^k$ on $G_k$.

In Step 3, we show $\{f^k(x)\}_{k\geq1}$ is bounded for each fixed $x\in V$ under (C-1) or (C-2). By choosing a subsequence of $f^k$, we get a solution to the Kazdan-Warner equation (\ref{kz-def-t-1}).\\

\textbf{Step 1.} Let $\{G_k\}_{k\geq1}$ be an exhaustion of $G$. When considering the condition (C-2) case, we require that $\{G_k\}_{k\geq1}$ satisfies Cheeger' inequality (\ref{cheeger-inequality}). For fixed $k$, consider the following equation
\begin{eqnarray}
\label{equ-lap-k}
  \Delta f(x)-g(x)+h(x)e^{f(x)}=0 \;\;\;\text{ for } x\in G_k, \notag\\
   f(x)=0    \;\;\;\text{ for } x\notin G_k.
\end{eqnarray}

Denote $\Delta_k$ as the discrete Laplacian on $G_k=(V_k,E_k)$, that is for each $x\in V_k$
$$\Delta_k f(x)=\frac{1}{\mu(x)} \sum_{y\thicksim x,\,y\in G_k}w_{xy}\big(f(y)-f(x)\big).$$
For each $x\in V_k$ we have
$$\Delta f(x)=\Delta_k f(x)-\varphi_k(x)f(x),$$
where
$$\varphi_k(x)=\frac{1}{\mu(x)} \sum_{y\thicksim x,\,y\notin G_k}w_{xy}.$$
Clearly $\varphi_k\geq 0$.
We rewrite the equation (\ref{equ-lap-k}) on $G_k$ as
\begin{equation}\label{kz-1}
\Delta_k f-\varphi_k f+he^f-g=0,\;\; \text{on } G_k,
\end{equation}
and call it the associated Kazdan-Warner equation on $G_k$.\\

Let $f(t,x):\mathbb{R}\times V_k\to\mathbb{R}$ be a function defined on $V_k$ at each time $t$, and varying with respect to $t$ smoothly. Consider the following heat equation
\begin{equation}
\label{heat}
\left(\frac{\partial}{\partial t}-\Delta_k\right)f=he^f-g-\varphi_kf.
\end{equation}


\begin{lemma}
\label{lem-heat-flow}
Given a function $f_0:V_k \to \mathbb{R}$. Assume $h\leq0$, then for each $k$, the solution to the heat equation (\ref{heat}) on $G_k$ with an initial condition $f(0,\cdot)=f_0$ exists on $[0,\infty)\times V_k$.
\end{lemma}
\begin{proof}
Write $f_t=\partial_tf$ and $f_{tt}=\partial_t\partial_tf$. Then the heat equation (\ref{heat}) can be written as an ODE system $f_t=F(t,f)$ with $F(t,f)=\Delta_kf+he^f-g-\varphi_kf$. Since all the coefficients in $F(t,f)$ (as a function of $f$) are
smooth and hence locally Lipschitz continuous. By the Picard theorem in classical ODE theory, the heat equation (\ref{heat}) on $G_k$ with an initial condition $f(0,\cdot)=f_0$ has a unique solution  on a small time interval $[0,\epsilon)$. By the extension theorem in classical ODE theory, there is a $0<T\leq+\infty$, such that the heat equation (\ref{heat}) has a solution $f$ on a right maximal time interval $[0, T)$. If $T\neq+\infty$, then $f_t(t,\cdot)$ ``blows up" when $t$ goes to $T$, i.e. $f_t$ can't be uniformly bounded for $t\in[0,T)$ and $x\in V_k$. We next show this will not happen.
By direct calculation, we obtain
\begin{eqnarray*}
\left(\Delta_k-\frac{\partial}{\partial t}\right)\big(f_t\big)^2&=&2  f_t \Delta_kf_t+2 |\nabla f|^2-2f_{tt}f_t\notag\\
&\geq&2f_t \frac{\partial}{\partial t}\big(\Delta_kf-f_t\big)\notag\\
&=&2\big(f_t\big)^2\big(-he^f+\varphi_k\big).
\end{eqnarray*}
Note $h\leq 0$ and $\varphi_k\geq0$, then
$$\left(\Delta_k -\frac{\partial}{\partial t}\right)\big(f_t\big)^2\geq 0.$$
By the maximum principle of parabolic equations, we obtain
\begin{eqnarray}
\label{heat-equ-ft-bound}
\max_{0\leq t<T,\,x\in V_k}|f_t(t,x)|^2&=&\max_{x\in V_k}|f_t(0,x)|^2\notag\\
&=&\max_{x\in V_k}|\Delta_kf_0+he^{f_0}-g-\varphi_kf_0|\notag\\
&=&C(G,h,g,f_0).
\end{eqnarray}
Hence $f_t(t,x)$ is uniformly bounded for $t\in[0,T)$ and $x\in V_k$. This means $T=+\infty$, which implies the conclusion.
\end{proof}

\textbf{Step 2.} Assume $h\in L^1(V)$ and $h\leq0$. Then we have
\begin{lemma}
\label{thm-bound-solution}
The associated Kazdan-Warner equation (\ref{kz-1}) has a solution $f^k$ on $V_k$ (if $h$ is not identically zero, we require $k$ is large enough so that $V_k$ contains a vertex $x_0$ with $h(x_0)\neq0$). Moreover, $f^k$ satisfies
\begin{equation}\label{est-k}
\int_{V_k}\left(f^kg+\frac{1}{2}\left(|\nabla f^k|^2+\varphi_k f^2\right)-\left(e^{f^k}-1\right)h\right)d\mu\leq 0.
\end{equation}
\end{lemma}
\begin{proof}
Consider the heat equation (\ref{heat}) on $G_k$ with an initial condition $f(0,\cdot)=0$. It has a unique solution $f$ by Lemma \ref{lem-heat-flow}. We show that $f$ is uniformly bounded on $[0,+\infty)\times V_k$. Consider the functional
$$J_k(f)=\int_{V_k} \left(fg+\frac{1}{2} |\nabla f|^2-(e^f-1)h+\frac{1}{2} \varphi_k f^2\right)d\mu.$$
Then by direct calculation, we get
\begin{eqnarray}
\label{Jk-deriv}
\frac{d}{dt}\big(J_k(f)\big)&=&\int_{V_k} \left(f_tg-\Delta_kff_t-he^{f}f_t+f_tf\varphi_k\right)d\mu\notag\\
&=&-\int_{V_k} (f _t)^2d\mu\leq 0.
\end{eqnarray}
Hence $J_k\big(f(t,\cdot)\big)$ is descending and
\begin{equation}
\label{est-pre}
\int_{V_k}\left(f g+\frac{1}{2}\left(|\nabla f|^2+\varphi_k f^2\right)-(e^{f}-1)h\right)d\mu\leq J_k\big(f(0,\cdot)\big)=0.
\end{equation}

We argue in the following of this step according to $h=0$ or $h\neq0$ on $V_k$.

(1) If $h=0$ on $V_k$. Extend $f(t,x)$ to $\overline{V_k}$ by setting $f(\cdot,x)=0$ at all vertex $x\in\partial V_k$, so
\begin{eqnarray}
\label{general-B-before}
\int_{\overline{V_k}} |\nabla f|^2d\mu&=&\sum_{x\thicksim y}w_{xy} \left(f(t,x)-f(t,y)\right)^2\notag\\
&=&\sum_{x,y\in V_k,x\thicksim y}w_{xy} \big(f(t,x)-f(t,y)\big)^2
+2\sum_{x\in V_k, y\in \partial {V_k}, x\thicksim y}w_{xy} f^2(t,x)\notag\\
&=&\int_{ V_k}\left(|\nabla f|^2+2\varphi_kf^2\right)d\mu\leq-2\int_{V_k}fgd\mu.
\end{eqnarray}
Note at the beginning of Step 1, we have made the assumption that the exhaustion $\{G_k\}$ satisfies Cheeger' inequality (\ref{cheeger-inequality}) when considering the condition (C-2) case, hence
\begin{eqnarray*}
\int_{V_k} |f|^2d\mu=\int_{\overline{V_k}}|f|^2d\mu&\leq& C(G)\int_{\overline{V_k}} |\nabla f|^2d\mu\\
&\leq&-2C(G)\int_{V_k}fgd\mu\\
&\leq&2C(G)\|g\|_{L^2(V_k)}\left(\int_{V_k}|f|^2d\mu\right)^{1/2}.
\end{eqnarray*}
It follows that there is a constant $C(G,g,V_k)$, such that for each $t\in\mathbb{R}$, 
$$\int_{V_k} |f(t,x)|^2d\mu\leq C(G,g,V_k).$$
This implies that $f(t,x)$ is uniformly bounded on $[0,+\infty)\times V_k$.

(2) If $h\neq0$ on $V_k$. We approach by contradiction. If $f$ is not bounded on $[0,+\infty)\times V_k$, then there is a time sequence $t_i\to+\infty$ such that
$$l_i=\max\limits_{G_k}|f (t_i, \cdot)|\to+\infty.$$
Set $u_i=\cfrac{f(t_i, \cdot)}{l_i}$, then by (\ref{est-pre}), we obtain
\begin{eqnarray*}
\int_{V_k}\left(fg+\frac{1}{2}|\nabla f|^2\right)d\mu&\leq&\int_{G_k}(e^f-1)hd\mu\notag\\
&\leq& \int_{G_k}(-h)d\mu\leq \int_{G} (-h)d\mu=\|h\|_{L^1(V)}.
\end{eqnarray*}
In particular,
$$\int_{V_k}\left(gu_il_i+\frac{1}{2}|\nabla u_i|^2l_i^2\right)d\mu\leq\|h\|_{L^1(V)}.$$
Since $l_i\to +\infty$, then by dividing $l_i^2$ on both sides at the same time, we get
$$\int_{V_k}|\nabla u_i|^2d\mu\to 0.$$
This implies that $u_i$ converges to some constant $C^*$ (=$1$ or $-1$). When $k$ is large enough, we can find $h(x_0)<0$ at some $x_0$ in $V_k$ (otherwise $h$ will be identically zero on $V_k$). In case
$C^*=1$, we have $e^{u_il_i}>1$ when $i$ is large enough, hence
\begin{eqnarray*}
\int_{V_k}(-h)\big(e^{f}-1\big)d\mu&=&\int_{V_k}(-h)\big(e^{u_il_i}-1\big)d\mu\\
&\geq&  -h(x_0)\big(e^{u_il_i}-1\big)\\
&>&\int_{V_k} -gu_il_id\mu=\int_{V_k} -gfd\mu,
\end{eqnarray*}
where the last inequality uses the fact that $e^x-1>ax$ if $x$ is big enough. This leads to a contradiction to (\ref{est-pre}). In case $C^*=-1$, note we have proved
$$|\Delta_k f+he^f-g-\varphi_k f|\leq C(G,h,g)$$
in the proof of Lemma \ref{lem-heat-flow}, see (\ref{heat-equ-ft-bound}). Substitute $f(t_i,\cdot)=u_il_i$ into the above inequality, and divide $l_i$ at both sides at the same time, we get
$$\Delta_k u_i+\frac{he^{u_il_i}}{l_i}-\frac{g}{l_i}-\varphi_k u_i\to 0$$
at each vertex of $V_k$. However, by direct calculation
$$\Delta_k u_i+\frac{he^{u_il_i}}{l_i}-\frac{g}{l_i}-\varphi_k u_i\to-\varphi_k.$$
So $\varphi_k =0$, this contradicts the fact that $G$ is connected. Thus we finally prove that $f$ is bounded on $[0,+\infty)\times V_k$. By (\ref{Jk-deriv}), the expression of $J'_k(t)$, we derive that there is a sequence of times $\{t_n\}$ going to $+\infty$, such that $f_t(t_n,x)\to 0$ at each vertex $x\in V_k$. Further note $f(t_n,\cdot)$ has a converging subsequence, which converges to some function $f^k(x),x\in V_k$. Obviously, $f^k$ is a solution to the associated Kazdan-Warner equation (\ref{kz-1}). Moreover, (\ref{est-k}) comes from (\ref{est-pre}).
\end{proof}

\textbf{Step 3.} If we can prove that $\{f^k(x)\}_{k\geq1}$ is bounded at each vertex $x\in V$, we get a global solution of the Kazdan-Warner equation (\ref{kz-def-t-1}) by choosing a subsequence of $f^k$. The reason is as follows. Note $f^k(x)=0$ if $x$ is not in $V_k$, then by choosing a subsequence we can get a function $f:V\to \mathbb{R}$ such that $f^k(x)\to f(x)$ at each fixed $x\in V$. For every fixed vertex $x\in V$ we have
$$\Delta_k f^k(x)+h(x)e^{f^k(x)}-g(x)-\varphi_k(x) f^k(x)=0.$$
If $k$ is large enough, then $\Delta_k f^k(x)=\Delta f^k(x)$ and $\varphi_k(x)=0$. Let $k\to+\infty$, we get
$$\Delta f(x)+h(x)e^{f(x)}-g(x)=0,$$
which implies that $f$ is a global solution of the Kazdan-Warner equation (\ref{kz-def-t-1}). So our main task in this step is to prove that $\{f^k(x)\}_{k\geq1}$ is bounded when either condition (C-1) or condition (C-2) is satisfied.

Assume the condition (C-1) is satisfied. We first prove $f^k\geq0$ on $V_k$. If this is not true, we may choose a vertex $x\in V_k$ such that $f^k(x)<0$ is the minimum value of $f^k$ on $V_k$. At the vertex $x$, we have
\begin{eqnarray*}
\Delta_k f^k(x)&=&g(x)-h(x)e^{f(x)}+\varphi_k(x)f^k(x)\\
&<&g(x)-h(x)\leq0.
\end{eqnarray*}
However, $\Delta_k f^k(x)\geq0$ since $f^k(x)$ is a minimum value. This leads to a contradiction. Hence $f^k\geq0$ on $V_k$.

Denote $\psi=\cfrac{g}{h}$. By (\ref{est-k}) and the fact $e^x-1>\frac{1}{2}x^2$ when $x>0$, we have
\begin{eqnarray*}
\int_{V_k}\frac{1}{2}\big(f^k\big)^2|h|d\mu&<&\int_{V_k}\big(e^{f^k}-1\big)|h|d\mu\\
&\leq&\int_{V_k}f^k|g|d\mu\\
&\leq&\left(\int_{V_k}\big(f^k\big)^2|h|d\mu\right)^{1/2}\left(\int_{V_k}\psi^2|h|d\mu\right)^{1/2}.
\end{eqnarray*}
Then it follows for each $k$
$$\int_{V_k}\big(f^k\big)^2d\mu\leq4\int_{V_k}|h|\psi^2d\mu\leq4\int_{V}|h|\psi^2d\mu.$$
This implies that $\{f^k(x)\}_{k\geq1}$ is bounded at each fixed vertex $x\in V$.\\

Assume the condition (C-2) is satisfied. Extend $f^k$ to $\overline{V_k}$ by setting $f^k=0$ on $\partial  V_k$. Similar
to (\ref{general-B-before}), we have
\begin{equation}
\label{general-B}
\int_{\overline{V_k}} |\nabla f^k|^2d\mu=\int_{ V_k}\left(|\nabla f^k|^2+2\varphi_k\big(f^k\big)^2\right)d\mu.
\end{equation}
Note $h\leq0$, hence $e^{f^k}h\leq0=h+|h|$ and then $\big(e^{f^k}-1\big)h\leq|h|$. Using this elementary inequality and (\ref{est-k}), we get
\begin{eqnarray}
\label{equ-no-number}
\int_{V_k}\left(f^kg+\frac{1}{2}|\nabla f^k|^2+\varphi_k \big(f^k\big)^2\right)d\mu&\leq&\int_{V_k}\big(e^{f^k}-1\big)hd\mu\notag\\
&\leq&\int_{V_k}|h|d\mu\leq \|h\|_{L^1(V)}.
\end{eqnarray}
Combining the above inequality (\ref{equ-no-number}) with the equality (\ref{general-B}), we get
\begin{equation*}
\int_{\overline{V_k}}|\nabla f^k|^2d\mu\leq2\|h\|_{L^1(V)}+2\int_{V_k}|f^kg|d\mu.
\end{equation*}
Since $G$ is a Cheeger graph, by Cheeger inequality (\ref{cheeger-inequality}) and $g\in L^2(V)$, we obtain
\begin{eqnarray*}
\int_{\overline{V_k}} |f^k|^2d\mu&\leq&C(G)\left(\|h\|_{L^1(V)}+\int_{V_k}|f^kg|d\mu\right)\\
&\leq&C(G)\|h\|_{L^1(V)}+C(G)\left(\int_{V_k}|f^k|^2d\mu\right)^{1/2}\left(\int_{V_k}g^2d\mu\right)^{1/2}\\
&\leq&C(G,h)+C(G)\|g\|_{L^2(V)}\left(\int_{V_k}|f^k|^2d\mu\right)^{1/2}.
\end{eqnarray*}
Therefore,
$$\int_{V_k} |f^k|^2d\mu=\int_{\overline{V_k}} |f^k|^2d\mu\leq C(G,h)+C(G,g)\left(\int_{V_k}|f^k|^2d\mu\right)^{1/2},$$
It follows that there is a constant $C(G,g,h)$, such that for each $k$,
$$\int_{V_k} |f^k|^2d\mu\leq C(G,g,h).$$
This implies that $\{f^k(x)\}_{k\geq1}$ is bounded at each fixed vertex $x\in V$. We finish the proof.\\

\noindent \textbf{Acknowledgements:} The authors would like to thank Professor Yanxun Chang, Gang Tian and Huijun Fan for constant encouragement. The research is supported by National Natural Science Foundation of China under Grant No.11501027.

Huabin Ge: hbge@bjtu.edu.cn

Department of Mathematics, Beijing Jiaotong University, Beijing, 100044, P.R. China\\[2pt]

Wenfeng Jiang: wen\_feng1912@outlook.com

School of Mathematical Sciences, Peking University, Beijing, 100871, P.R. China

\end{document}